\documentclass[10pt]{article}
   \usepackage{amsmath,amsfonts,amsthm,amssymb,amscd,enumerate, url, anysize}
   \usepackage[pdftex]{graphicx,color}     
\marginsize{1in}{1in}{1in}{1in}

\usepackage{ifthen}
\usepackage{ifpdf}
\usepackage{euscript}
\usepackage{varioref}
\usepackage[all]{xy}
\xyoption{web}
\ifpdf
    \usepackage[pdftex,pagebackref]{hyperref}
\else \fi
\newtheorem{theorem}{Theorem}[section]

\newtheorem{lemma}[theorem]{Lemma}
\newtheorem{proposition}[theorem]{Proposition}
\theoremstyle{remark}

\newtheorem{definition}[theorem]{Definition}

\newtheorem{remark}[theorem]{Remark}

\newtheorem{assumption}[theorem]{Assumption}
\newcounter{remarkCounter}
\newcommand{\pd}[2]{\frac{\partial #1}{\partial #2}}
\newlength{\setBracketHeight}

\newcommand{\LieDer}{\ensuremath{\EuScript L}}
\newcommand{\hook}{\mathbin{\! \hbox{\vrule height0.3pt width5pt depth 0.2pt \vrule height5pt width0.4pt depth 0.2pt}}}
\newcommand{\R}[1]{\ensuremath{\mathbb{R}^{#1}}}
\newcommand{\C}[1]{\ensuremath{\mathbb{C}^{#1}}}

\newcommand{\vol}[1]{\ensuremath{\mathsf{vol}_{#1}}}

\newcommand{\nForms}[2]{\ensuremath{\Omega^{#1} \! \left({#2}\right)}}

\newcommand{\st}[1]{\ensuremath{\ast_{ {}_{\! #1}}}}
\renewcommand{\Re}[1]{\ensuremath{\operatorname{Re} #1}}
\renewcommand{\Im}[1]{\ensuremath{\operatorname{Im} #1}}
\newcommand{\SU}[1]{\ensuremath{\mathrm{SU}\! \left(#1\right)}}
\newcommand{\G}{\ensuremath{\mathrm{G}_2}}
\newcommand{\SP}{\ensuremath{\mathrm{Spin}(7)}}
\newcommand{\ph}{\ensuremath{\varphi}}
\newcommand{\ps}{\ensuremath{\psi}}

\numberwithin{equation}{section}
\numberwithin{table}{section}
\numberwithin{figure}{section}

\usepackage{cite}

\begin{document}

\title{Soliton solutions for the Laplacian coflow \\ of some $\G$-structures with symmetry}

\author{Spiro Karigiannis\footnote{The research of the first author is partially supported by a Discovery Grant from the Natural Sciences and Engineering Research Council of Canada.} \\ {\it Department of Pure Mathematics, University of Waterloo} \\ \tt{karigiannis@math.uwaterloo.ca} \and Benjamin McKay\footnote{The research of the second author is supported by Grant No. MATF636 of Science Foundation Ireland.} \\ {\it University College Cork, National University of Ireland} \\ \tt{B.McKay@ucc.ie} \and Mao-Pei Tsui \\ {\it Department of Mathematics, University of Toledo} \\ \tt{Mao-Pei.Tsui@utoledo.edu} }

\maketitle

\begin{abstract}
We consider the Laplacian ``co-flow'' of $\G$-structures: $\pd{}{t} \ps = - \Delta_d \ps$ where $\ps$ is the dual $4$-form of a $\G$-structure $\ph$ and $\Delta_d$ is the Hodge Laplacian on forms. Assuming short-time existence and uniqueness, this flow preserves the condition of the $\G$-structure being coclosed ($d\ps =0$). We study this flow for two explicit examples of coclosed $\G$-structures with symmetry. These are given by warped products of an interval or a circle with a compact $6$-manifold $N$ which is taken to be either a nearly K\"ahler manifold or a Calabi-Yau manifold. In both cases, we derive the flow equations and also the equations for soliton solutions. In the Calabi-Yau case, we find all the soliton solutions explicitly. In the nearly K\"ahler case, we find several special soliton solutions, and reduce the general problem to a single \emph{third order} highly nonlinear ordinary differential equation.
\end{abstract}

\tableofcontents

\section{Introduction}

Flows of $\G$-structures were first considered by Robert Bryant in~\cite{Bryant:2006}. In particular, Bryant considered the Laplacian flow of $\G$-structures: $\pd{}{t} \ph = \Delta_d \ph$, where $\ph$ is a non-degenerate $3$-form defining a $\G$-structure, and $\Delta_d$ is the Hodge Laplacian on forms. In the case when $\ph$ is closed, this condition is preserved under the flow. Using an appropriate choice of inner product on the space of exact $3$-forms, one can also show that this flow is the \emph{gradient flow} for the volume functional on the space of torsion-free $\G$-structures which was introduced by Hitchin in the arXiv version of~\cite{Hitchin:2000}.

\begin{remark}
Note that since the Hodge Laplacian $\Delta_d$ is equal to \emph{minus} the rough Laplacian $\nabla^* \nabla$ plus lower order terms (by the Weitzenb\"ock formula), it can be argued that it is more natural to consider $\pd{}{t} \ph = - \Delta_d \ph$ in order for this flow to be qualitatively like a heat equation. However, for \emph{closed} $\G$-structures, one can show that $\Delta_d \ph$ actually only contains first derivatives of $\ph$, so that $\Delta_d \ph$ and $-\Delta_d \ph$ are the same, up to lower order terms. Therefore in this case only, both flows are heat-like. The choice $+ \Delta_d \ph$ has the advantage that it is the gradient flow for the Hitchin functional, so it does increase the volume along the flow, and the torsion-free $\G$-structures are indeed local maxima of the Hitchin volume functional. The fact that $\Delta_d \ph$ contains only first derivatives of $\ph$ when $\ph$ is closed can be shown using the general machinery for flows of $\G$-structures in~\cite{Karigiannis:2009}.
\end{remark}

Since this fundamental work by Bryant, the first author has developed several formulas for general flows of $\G$-structures in~\cite{Karigiannis:2009}. More recently, there has been work by Xu--Ye~\cite{Xu:Ye:2010}, Weiss--Witt~\cite{Weiss:Witt:2010} and Bryant--Xu~\cite{Bryant:Xu:2011} on the short-time existence and uniqueness of solutions for the Laplacian flow $\pd{}{t} \ph = \Delta_d \ph$ for closed $\G$-structures.

In this paper we consider the Laplacian ``coflow'' of $\G$-structures, by which we mean the Laplacian flow of the dual $4$-form $\ps = \st{\ph} \ph$. That is, $\pd{}{t} \ps = - \Delta_d \ps$. Since this flow cannot be related to the Hitchin volume functional in any obvious way, and because we do not restrict ourselves to closed $\G$-structures, but rather to coclosed $\G$-structures, we include a minus sign in front of our Hodge Laplacian to make the equation heat-like. If we assume short-time existence and uniqueness, then this flow preserves the coclosed ($d\ps=0$) condition, as we discuss in Section~\ref{coflowsec}. The reason we consider this flow is because there exists a general ansatz for a cohomogeneity one $\G$-structure on $M^7 = N^6 \times L^1$ which is a \emph{coclosed} $\G$-structure. Here we take the $1$-manifold $L^1$ to be either $\R{}$ or $S^1$, and the compact $6$-manifold $N^6$ is taken to be a nearly K\"ahler $6$-manifold or a Calabi-Yau $3$-fold.

In Section~\ref{reviewsec} we review various facts about $\G$-structures and their torsion forms. In Section~\ref{su3sec} we discuss $\SU{3}$-structures on a $6$-manifold $N^6$, and focus on the special cases of Calabi-Yau and nearly K\"ahler structures. We also develop some formulas we need later. Section~\ref{coflowsec} discusses certain properties of the Laplacian coflow, including its associated soliton solutions. Finally in Sections~\ref{CYsec} and~\ref{NKsec} we explicitly derive the evolution equations and soliton equations (and discuss their solutions) when $N^6$ is Calabi-Yau or nearly K\"ahler, respectively. In particular, we find all the soliton solutions in the Calabi-Yau case. In the nearly K\"ahler case, we find several special solutions to the soliton equations, including the interesting case of a \emph{sine-cone} metric over a nearly K\"ahler manifold, which corresponds to a non-torsion free $\G$-structure that is an eigenform of its own Laplacian.

Cohomogeneity one solitons for the Ricci flow have been extensively studied. Some references (this list is not exhaustive) include~\cite{Dancer/Wang:2000, Dancer/Wang:2005, Petersen/Wylie}.

{\em Note:} Throughout this paper, we use $| \cdot |$ and $\langle \cdot, \cdot \rangle$ to denote the pointwise norm and inner product on differential forms and $|| \cdot ||$ and $\langle \langle \cdot, \cdot \rangle \rangle$ to denote the $L^2$ norm and inner product on forms (the integral of the corresponding pointwise quantity over the manifold.)

{\bf Acknowledgements.} The first author would like to acknowledge useful discussions with Christopher Lin and Feng Xu. The second author would like to acknowledge useful discussions with Robert L.\ Bryant. The third author would like to thank the Department of Pure Mathematics at the University of Waterloo, which he was visiting when parts of this work were completed. The authors also thank the referee for useful suggestions and for alerting us to some important references that we had overlooked.

\section{Review of $\G$-structures and their torsion} \label{reviewsec}

We begin by recalling the definition of a \(\G\)-structure.

\begin{definition} \label{g2structuredefn}
A \(3\)-form \(\ph\) on a \(7\)-manifold \(M^7\) is called \emph{nondegenerate} if for any nonzero  \( X \in T_p M\),
\[
0 \ne \left(X \hook \ph \right) \wedge \left(X \hook \ph \right) \wedge \ph.
\]
A smooth nondegenerate \(3\)-form is also called a \emph{\(\G\)-structure}. If \(\ph\) is a \(\G\)-structure, then there is a unique metric \(g=g_{\ph}\) and orientation such that if \(\vol{} = \vol{\ph}\) is the volume form associated to that metric and orientation, then for any point \(p \in M\) and any vectors \(X, Y \in T_p M\), we have
\[
-\frac{1}{6}\left(X \hook \ph \right) \wedge \left(Y \hook \ph \right) \wedge \ph = g(X,Y) \vol{\ph}.
\]
See Bryant \cite{Bryant:2006} for a proof. Note that we are using the opposite orientation of \cite{Bryant:1987b,Bryant:2006}. Let \(\st{\ph}\) be the Hodge star operator of \(g_{\ph}\) with the orientation induced by \(\ph\). We will often write \(\st{\ph}\) as \(\st{7}\) to indicate the dimension of the manifold \(M^7\). We will always write \(\ps\) to mean the dual \(4\)-form \(\ps=\st{\ph} \ph\).
\end{definition}

There are various natural conditions on \(\G\)-structures that one can consider.

\begin{definition} \label{g2structuresconditionsdefn}
A \(\G\)-structure \(\ph\) is called \emph{closed} if \(d \ph=0,\) \emph{coclosed} if \(d \ps=0,\) and torsion-free if \(\Delta_d \ph = 0\) (or equivalently if \(\Delta_d \ps=0\)).
\end{definition}

The space of forms on $M^7$ decomposes into irreducible subspaces under the action of \(\G\), and this allows us to define the \emph{torsion forms} of a \(\G\)-structure. In particular we have $\Lambda^4 = \Lambda^4_1 \oplus \Lambda^4_7 \oplus \Lambda^4_{27}$ and $\Lambda^5 = \Lambda^5_7 \oplus \Lambda^5_{14}$. Precise descriptions of these subspaces, which we will not require here, can be found in~\cite{Bryant:2006, Joyce:2000, Karigiannis:2009}.

\begin{definition} \label{torsiondefn}
If \(\ph\) is a \(\G\)-structure on a \(7\)-manifold, with associated \(4\)-form \(\ps\), then there are unique forms \(\tau_0, \tau_1, \tau_2, \tau_3,\) called the \emph{torsion forms} of the \(\G\)-structure, where \(\tau_k\) is a \(k\)-form, such that
\begin{equation} \label{torsionformseq}
\begin{aligned}
d \ph &= \tau_0\ps + 3 \tau_1 \wedge \ph + \st{\ph} \tau_3, \\
d \ps &= 4 \tau_1 \wedge \ps + \st{\ph} \tau_2.
\end{aligned}
\end{equation}
We can recover the torsion forms using the following identities:
\begin{align} \label{tauzeroeq}
\tau_0 &= \frac{1}{7} \st{7} \left( \ph \wedge d \ph \right) \\ \label{tauoneeq}
\tau_1 &= \frac{1}{12} \st{7} \left( \ph \wedge \st{7} d \ph \right) = \frac{1}{12} \st{7} \left( \ps \wedge \st{7} d \ps \right)
\end{align}
\end{definition}
See~\cite{Bryant:2006} or~\cite{Karigiannis:2009} for a more detailed discussion about the torsion forms, including the derivation of the above equations. The torsion forms were first considered by Fern\`andez--Gray~\cite{Fernandez:Gray:1982} and are also discussed in detail in~\cite{Cabrera:1996} and~\cite{Friedrich:Ivanov:2002}, for example. When the four torsion forms vanish (equivalently when $\ph$ is closed and coclosed) the $\G$-structure is called torsion-free and it can be shown that the Riemannian holonomy of the metric $g_{\ph}$ is contained in $\G$, and that $g_{\ph}$ is Ricci-flat.

\section{$\mathrm{SU}(3)$-structures and their associated $\G$-structures} \label{su3sec}

Let $N^6$ be a smooth $6$-manifold. An $\SU{3}$-structure on $N^6$ is a reduction of the structure group from $\mathrm{GL}(6, \R{})$ to $\SU{3}$. Such manifolds come equipped with an almost complex structure $J$, a Riemannian metric $g$ with respect to which $J$ is orthogonal, and a particular choice of nowhere vanishing smooth complex-valued $3$-form $\Omega$ of type $(3,0)$. The metric and the almost complex structure together determine the K\"ahler form $\omega(X,Y) = g(JX, Y)$, which is a real $2$-form of type $(1,1)$. At each point on $N$, the magnitude of $\Omega$ can be fixed by the requirement that these structures are related by the following equation:
\begin{equation} \label{SU3relationeq}
\vol{N} = \frac{\omega^3}{3!} = \frac{i}{8}  \Omega \wedge \bar{\Omega} = \frac{1}{4} \Re{(\Omega)} \wedge \Im{(\Omega)}.
\end{equation}
Note that if we change $\Omega$ to $e^{i\theta} \Omega$, for some phase function $e^{i\theta}$ which can vary on $N$, then we get the same $\mathrm{U}(3)$-structure but a different $\SU{3}$-structure.
\begin{remark} \label{SU3coframermk}
For a manifold \(N^6\) equipped with an \(\SU{3}\)-structure, near each point of \(N^6\) we can find a local unitary coframe of complex-valued $1$-forms \((\xi_1, \xi_2, \xi_3)\) for which
\begin{align*}
\omega &= \frac{i}{2} \sum_p \xi_p \wedge \bar \xi_p, \\
\Omega &= \xi_1 \wedge \xi_2 \wedge \xi_3.
\end{align*}
It is clear that these forms are independent of the choice of such local unitary coframe, as long as it maintains the same ``complex orientation.'' This means that the two frames can only differ by an element of $\SU{3}$ at each point on $N$.
\end{remark}

We will write the Hodge star operator of \(N\) as \(\st{6}\), the metric as \(g_{6}\) and the volume form as \(\vol{6}\). It is then easy to check the following identities (which will be employed often in later sections):
\begin{equation} \label{SU3relationseq}
\begin{aligned}
\st{6}^2 &= (-1)^k \text{ on } \nForms{k}{N}, & \st{6} \Omega &= -i\Omega, & \st{6} \bar{\Omega} &= i\bar{\Omega} \\ \st{6} \omega &= \frac{\omega^2}{2}, & \st{6} \frac{\omega^2}{2} &= \omega, &  \left( x \hook \Omega\right) \wedge \omega &= \Omega \wedge \left( x \hook \omega\right).
\end{aligned}
\end{equation}

The importance of $\SU{3}$-structures for our purposes is that they naturally induce $\G$-structures on $M^7 = N^6 \times L^1$, where $L^1$ can be $\R{}$ or $S^1$. Let $r$ be a local coordinate on $L^1$. Then the $3$-form $\ph$ defined by $\ph = \Re(\Omega) - dr \wedge \omega$ is a $\G$-structure on $M^7$, inducing the product metric $g_7 = dr^2 + g_6$ and the dual $4$-form $\ps = -dr \wedge \Im{\Omega} - \frac{\omega^2}{2}$.
See~\cite{Karigiannis:Notes} for a detailed discussion of this relationship, as well as an explanation of the different sign conventions for $\G$-structures. The relationships between $\SU{3}$-structures and $\G$-structures are also discussed in~\cite{Chiossi:Salamon:2001} and~\cite{Cabrera:2006}.

\begin{definition} \label{symmetryg2structuredefn}
We can define a more general $\G$-structure on $M^7$ which is cohomogeneity one with respect to the $\SU{3}$ action. Let $F(r)$ be a smooth, nowhere vanishing complex-valued function on $L^1$, and let $G(r)$ to be a smooth, everywhere positive function on $L^1$. Then
\begin{equation} \label{symmetrypheq}
\ph = \Re{(F^3 \Omega)} - G {|F|}^2 dr \wedge \omega
\end{equation}
is a $\G$-structure on $M^7$, with induced metric
\begin{equation} \label{symmetrymetriceq}
g_7 = G^2 dr^2 + {|F|}^2 g_6,
\end{equation}
associated volume form
\begin{equation} \label{symmetryvoleq}
\vol{7} = G {|F|}^6 dr \wedge \vol{6},
\end{equation}
and dual $4$-form
\begin{equation} \label{symmetrypseq}
\ps = - G dr \wedge \Im{(F^3 \Omega)} - {|F|}^4 \frac{\omega^2}{2}.
\end{equation}
With regards to the $\SU{3}$ local unitary coframe on $N$ described in Remark~\ref{SU3coframermk}, this simply corresponds to choosing $\{F\Re{\xi_1}, F\Re{\xi_2}, F\Re{\xi_3}, F\Im{\xi_1}, F\Im{\xi_2}, F\Im{\xi_3}, Gdr\}$ to be an orthonormal $\G$-adapted coframe for $M^7$. 
\end{definition}

\begin{remark} \label{Gremark}
We remark that the function $G(r)$ can always be set equal to $1$ by defining a new local coordinate to be $\tilde r = \int_0^r G(s) ds$, so $d \tilde r = G(r) dr$. \emph{However}, when we are considering a flow of $\G$-structures $\ph(t)$, it will be convenient to include this factor of $G(r)$, because then $G(r)$ and thus the change of variables $\tilde r = \tilde r(r)$ will in general also be $t$-dependent. This will become clear in Section~\ref{coflowsec}.
\end{remark}

\subsection{Calabi-Yau threefolds} \label{CYintrosec}

When both the K\"ahler form $\omega$ and the nonvanishing $(3,0)$ form $\Omega$ are \emph{parallel} with respect to the Levi-Civita connection $\nabla$ of the metric $g$, then $(N^6, g, \omega, \Omega)$ is called a \emph{Calabi-Yau threefold}. In particular the forms $\omega$ and $\Omega$ are both closed: $d\omega=0$ and $d\Omega = 0$. See~\cite{Joyce:2000} for more about the differential geometry of Calabi-Yau manifolds. In this case, the ansatz given by equations~\eqref{symmetrypheq} and~\eqref{symmetrypseq} for the $\G$-structure on $N^6 \times L^1$ will be torsion-free (closed and coclosed) if and only if
\begin{align*}
d \left( \frac{1}{2} F^3 \Omega + \frac{1}{2} \bar F^3 \bar \Omega - G F \bar F dr \wedge \omega \right) \, & = \, \frac{3}{2} \left( F^2 F' dr \wedge \Omega + \bar F^2 \bar F' dr \wedge \bar \Omega \right) \, = 0, \\ d \left( -\frac{1}{2i} G F^3 dr \wedge \Omega + \frac{1}{2i} G \bar F^3 dr \wedge \bar \Omega - F^2 \bar F^2 \frac{\omega^2}{2} \right) \, & = \, -2 \left( F F' \bar F^2 + F^2 \bar F \bar F' \right) dr \wedge\frac{\omega^2}{2} \, = 0. \\
\end{align*}
By comparing types, these equations are satisfied if and only if $F' = 0$. Hence $F$ must be constant for the $\G$-structure to be torsion-free. By remark~\ref{Gremark}, in the time-independent case we can always assume $G=1$, and by rescaling the $\SU{3}$-structure on $N$ we can assume that $F=1$ also. Hence $M^7$ is then metrically a product of the Calabi-Yau $3$-fold and the standard flat metric on $L^1$.

\subsection{Nearly K\"ahler 6-manifolds} \label{NKintrosec}

Another interesting $\SU{3}$-structure that is related to $\G$-geometry is that of a \emph{nearly K\"ahler $6$-manifold}. In this case, the forms $\omega$ and $\Omega$ satisfy the following system of coupled equations:
\begin{equation} \label{NKequations}
\begin{aligned}
d \omega & = - 3 \Re{(\Omega)}, & & & d \Re{(\Omega)} & = 0, \\
d \Im{(\Omega)} & = 4 \, \frac{\omega^2}{2}, & & & d \, \frac{\omega^2}{2} & = 0.
\end{aligned}
\end{equation}
Of course the second column of equations in~\eqref{NKequations} follow immediately from the first column, but we prefer to list them all together as we will require them all for computations in Section~\ref{torsionsec}.

An excellent survey of nearly K\"ahler manifolds is~\cite{Reyes:Salamon:1999}. We remark that, other than the standard round $S^6$, only three other examples of compact nearly K\"ahler $6$-manifolds are known, and these are all homogeneous spaces. The fact that these are the only compact homogeneous examples that can exist was proved by Butruille~\cite{Butruille:2005}. It is expected, however, that there should exist many non-homogeneous compact examples. The case of cohomogeneity-one complete nearly K\"ahler manifolds has been studied by Podest\`a--Spiro in~\cite{Podesta:Spiro:1:2010} and~\cite{Podesta:Spiro:2:2010}.

For the purposes of the present paper, we will only need to use the equations~\eqref{NKequations} describing a nearly K\"ahler $6$-manifold, in addition to the standard relations of an $\SU{3}$-structure from equations~\eqref{SU3relationeq} and~\eqref{SU3relationseq}. In this case, the ansatz~\eqref{symmetrypheq} and~\eqref{symmetrypseq} for the $\G$-structure on $N^6 \times L^1$ will be torsion-free (closed and coclosed) if and only if
\begin{align*}
d \ph \, & = \, \frac{3}{2} \left( F^2 F' dr \wedge \Omega + \bar F^2 \bar F' dr \wedge \bar \Omega \right) + \frac{1}{2} F^3 (4 i) \frac{\omega^2}{2} + \frac{1}{2} \bar F^3 (-4 i) \frac{\omega^2}{2} + G F \bar F dr \wedge \left( -\frac{3}{2} \Omega - \frac{3}{2} \bar \Omega \right) \\ & = \, \frac{3}{2} \left( F^2 F' - G F \bar F \right) dr \wedge \Omega + \frac{3}{2} \left( \bar F^2 \bar F' - G F \bar F \right) dr \wedge \bar \Omega + 2i (F^3 - \bar F^3) \frac{\omega^2}{2} \, = \, 0,
\end{align*}
and
\begin{align*}
d \ps \, & = \, -2 \left( F F' \bar F^2 + F^2 \bar F \bar F' \right) dr \wedge\frac{\omega^2}{2} + \frac{1}{2i} G F^3 dr \wedge (4i) \frac{\omega^2}{2} - \frac{1}{2i} G \bar F^3 dr \wedge (-4i) \frac{\omega^2}{2} \\ & = \, \left( 2G(F^3 + \bar F^3 ) - 2(F^2 \bar F \bar F' + \bar F^2 F F') \right) dr \wedge \frac{\omega^2}{2} \, = \, 0.
\end{align*}
Again assuming that $G=1$, it is easy to check that the solution to this system of equations is $F(r) = r$, yielding the metric
\begin{equation*}
g_7 = dr^2 + r^2 g_6
\end{equation*}
which is a \emph{metric cone} over the space $N^6$. Here we need to take $L^1 = (0, \infty)$. In fact, one can also \emph{define} nearly K\"ahler $6$-manifolds to be exactly those spaces for which the Riemannian cone over them has holonomy contained in $\G$. See B\"ar~\cite{Bar:1993} for details.

\begin{remark}
See also Cleyton--Swann~\cite{Cleyton:Swann:2002} for another application of $\SU{3}$-structures to cohomogeneity one $\G$-structures.
\end{remark}

\subsection{Some invariant formulas on $M^7 = N^6 \times L^1$} \label{formulassec}

In this section we collect together some formulas involving the Hodge star operators $\ast_6$ and $\ast_7$ on $N^6$ and $M^7$, respectively, which we will use in both the Calabi-Yau and the nearly K\"ahler cases to study the Laplacian coflow. We also discuss the Laplacian and gradient for functions on $M^7$ which depend only on the coordinate $r$ on $L^1$, which we will need later to express the evolution and soliton equations in an invariant form.

We consider the ansatz~\eqref{symmetrypheq} for a cohomogeneity one $\G$-structure on $M^7$. To simplify the calculations somewhat, we will sometimes write
\begin{equation*}
F = h e^{i \theta}
\end{equation*}
for some smooth \emph{real valued} functions $h$ and $\theta$ on $L^1$. Hence we can write equations~\eqref{symmetrypheq} and~\eqref{symmetrypseq} as
\begin{equation} \label{phpseqs}
\begin{aligned}
\ph & = \frac{F^3}{2} \Omega + \frac{\bar F^3}{2} \bar \Omega - G h^2 dr \wedge \omega, \\
\ps & = \frac{i G F^3}{2} dr \wedge \Omega - \frac{i G \bar F^3}{2} dr \wedge \bar \Omega - h^4 \frac{\omega^2}{2}.
\end{aligned}
\end{equation}
and the metric and volume form as
\begin{equation} \label{metricvoleq}
g_7 = G^2 dr^2 + h^2 g_6, \qquad \qquad \vol{7} = G h^6 dr \wedge \vol{6}.
\end{equation}
Using~\eqref{metricvoleq} for the metric and the volume form on $M^7$, it is easy to see that if $\alpha$ is any $k$-form on $N^6$, then we have
\begin{equation} \label{symmetrystareq}
\begin{aligned}
\st{7} \alpha & = \, (-1)^k h^{6-2k} G \, dr \wedge \st{6} \alpha, \\
\st{7} \left( dr \wedge \alpha \right) & = \, h^{6-2k} G^{-1} \st{6} \alpha.
\end{aligned}
\end{equation}
Using these equations and~\eqref{SU3relationseq}, we find that
\begin{equation} \label{symmetrystareq2}
\left.
\begin{aligned}
\st{7} \omega &= h^2 G dr \wedge \frac{\omega^2}{2}, & & & \st{7} \left(dr \wedge \omega \right) & = h^2 G^{-1} \frac{\omega^2}{2}, \\
\st{7} \Omega &= iG dr \wedge \Omega, & & & \st{7} \left(dr \wedge \Omega\right) & = -iG^{-1} \Omega, \\
\st{7} \left( \frac{\omega^2}{2} \right) & = h^{-2} G \, dr \wedge \omega, & & &  \st{7} \left(dr \wedge \frac{\omega^2}{2} \right) & = h^{-2} G^{-1} \omega. \\
\end{aligned}
\qquad \qquad \right\}
\end{equation}

\begin{remark}
Throughout this paper, we will always use a prime $'$ to denote differentiation with respect to the coordinate $r$ on $L^1$.
\end{remark}

Suppose that $f = f(r)$ is a function depending only on the coordinate $r$ on $L^1$. Then using~\eqref{symmetrystareq} we can compute that its Hodge Laplacian $\Delta_d f$ is given by
\begin{align*}
\Delta_d f & = \, d^* d f \, = \, - \st{7} \! d \st{7} \! f' dr \, = \, - \st{7} \! d ( f' \st{7} \! dr ) \, = \, - \st{7} \! d \left( \frac{f' h^6}{G} \vol{6} \right) \\ & = \, - \st{7} \! \left( \left( \frac{f' h^6}{G} \right)' dr \wedge \vol{6} \right) \, = \, - \frac{1}{G h^6} \left( \frac{f' h^6}{G} \right)'
\end{align*}
\begin{remark}
We will use the symbol $\Delta$ (without the $d$ subscript) to denote the \emph{rough Laplacian} $\nabla^* \nabla$, which, on functions, differs from $\Delta_d$ by a sign.
\end{remark}
Hence the above equation gives
\begin{equation} \label{roughlaplacianeq}
\Delta f \, = \,  \frac{1}{G h^6} \left( \frac{f' h^6}{G} \right)' \, = \, \frac{f''}{G^2} + \frac{6 h' f'}{h G^2} - \frac{f' G'}{G^3}.
\end{equation}
We also note by~\eqref{metricvoleq}, if $f = f(r)$ and $\rho = \rho(r)$, and $\nabla$ denotes the \emph{gradient} with respect to $g_7$, then we have that
\begin{equation} \label{gradienteq}
\langle \nabla f , \nabla \rho \rangle \, = \, \langle df, d\rho \rangle \, = \, f' \rho' \langle dr, dr \rangle \, = \, \frac{f' \rho'}{G^2}, \qquad \qquad |\nabla f|^2 \, = \, \frac{(f')^2}{G^2}.
\end{equation}

\subsection{The torsion forms} \label{torsionsec}

In this section we compute the four torsion forms $\tau_0$, $\tau_1$, $\tau_2$, and $\tau_3$ that we defined in Section~\ref{reviewsec} for our cohomogeneity one $\G$-structure on $N^6 \times L^1$, in the two cases where $N^6$ is either Calabi-Yau or nearly K\"ahler. Differentiating the forms in~\eqref{phpseqs} gives
\begin{align*}
d\ph & = \frac{(F^3)'}{2} dr \wedge \Omega + \frac{(\bar F^3)'}{2} dr \wedge \bar \Omega + G h^2 dr \wedge d \omega + \frac{F^3}{2} d \Omega + \frac{\bar F^3}{2} d \bar \Omega, \\
d\ps & = -\frac{i G F^3}{2} dr \wedge d\Omega + \frac{i G \bar F^3}{2} dr \wedge d\bar \Omega - (h^4)' dr \wedge \frac{\omega^2}{2} - h^4 d\left( \frac{\omega^2}{2} \right).
\end{align*}
In the Calabi-Yau case, we have $d\omega = 0$ and $d\Omega = 0$, while in the nearly K\"ahler case, equations~\eqref{NKequations} say
\begin{equation} \label{NKequationscomplex}
d\omega = -\frac{3}{2} \Omega - \frac{3}{2} \bar \Omega, \qquad \qquad d \left(\frac{\omega^2}{2}\right) = 0, \qquad \qquad d \Omega = 4i \frac{\omega^2}{2}.
\end{equation}
Therefore,
\begin{equation} \label{firstderivativeeqs}
\left.
\begin{aligned}
& \text{when $N^6$ is Calabi-Yau:} \\ & \qquad d\ph = \frac{(F^3)'}{2} dr \wedge \Omega + \frac{(\bar F^3)'}{2} dr \wedge \bar \Omega, \\ & \qquad d\ps = - (h^4)' dr \wedge \frac{\omega^2}{2}; \\ & \text{when $N^6$ is nearly K\"ahler: } \\ & \qquad d\ph = \left( \frac{(F^3)'}{2} - \frac{3}{2} G h^2 \right) dr \wedge \Omega + \left( \frac{(\bar F^3)'}{2} - \frac{3}{2} G h^2 \right) dr \wedge \bar \Omega + 2 i (F^3 - \bar F^3) \frac{\omega^2}{2}, \\ & \qquad d\ps = \left( 2G (F^3 + \bar F^3) - (h^4)' \right)dr \wedge \frac{\omega^2}{2}.
\end{aligned}
\quad \quad \right\}
\end{equation}
Using the identities of~\eqref{symmetrystareq2} we immediately get
\begin{equation} \label{starfirstderivativeeqs}
\left.
\begin{aligned}
& \text{when $N^6$ is Calabi-Yau:} \\ & \qquad \st{7} \! (d\ph) = -\frac{i(F^3)'}{2G} \Omega + \frac{i(\bar F^3)'}{2G} \bar \Omega, \\ & \qquad \st{7} \! (d\ps) = - \frac{(h^4)'}{G h^2} \omega; \\ & \text{when $N^6$ is nearly K\"ahler: } \\ & \qquad \st{7} \! (d\ph) = - \frac{i}{G} \left( \frac{(F^3)'}{2} - \frac{3}{2} G h^2 \right) \Omega + \frac{i}{G} \left( \frac{(\bar F^3)'}{2} - \frac{3}{2} G h^2 \right) \bar \Omega + \frac{2 i G}{h^2} (F^3 - \bar F^3) dr \wedge \omega, \\ & \qquad \st{7} \! (d\ps) = \frac{1}{G h^2} \left( 2G (F^3 + \bar F^3) - (h^4)' \right) \omega.
\end{aligned}
\quad \right\}
\end{equation}
We are now in a position to compute the torsion forms of these $\G$-structures.

\begin{lemma} \label{torsionlemma}
For such a $\G$-structure, the zero-torsion $\tau_0$ and the one-torsion $\tau_1$ are as follows.
\begin{equation} \label{tauzerotauoneeqs}
\left.
\begin{aligned}
& \text{when $N^6$ is Calabi-Yau:} && \tau_0 = \frac{12}{7G} \theta', & & \qquad \tau_1 = d (\log h); \\ & \text{when $N^6$ is nearly K\"ahler: } && \tau_0 = \frac{12}{7} \left( \frac{\theta'}{G} + \frac{2 \sin 3 \theta}{h} \right), & & \qquad \tau_1 = \left(\frac{h' - G \cos 3 \theta}{h} \right) dr.
\end{aligned}
\quad \quad \right\}
\end{equation}
and the two-torsion $\tau_2$ always vanishes:
\begin{equation} \label{tautwoeq}
\tau_2 \, = \, 0
\end{equation}
in both the Calabi-Yau and the nearly K\"ahler cases.
\end{lemma}
\begin{proof}
Substitute equations~\eqref{firstderivativeeqs} and~\eqref{starfirstderivativeeqs} into the formulas~\eqref{tauzeroeq} and~\eqref{tauoneeq} for the zero-torsion $\tau_0$ and the one-torsion $\tau_1$, and use~\eqref{symmetrystareq2}. It is a tedious but straightforward computation to obtain~\eqref{tauzerotauoneeqs}. Now equations~\eqref{torsionformseq} can be solved for the two-torsion $\tau_2$ and the three-torsion $\tau_3$:
\begin{align*}
\tau_2 & = \, \st{7} \! (d \ps) - 4 \st{7} \! ( \tau_1 \wedge \ps), \\
\tau_3 & = \, \st{7} \! (d \ph) - \tau_0 \ph - 3 \st{7} \! ( \tau_1 \wedge \ph).
\end{align*}
From these we can obtain an explicit (albeit complicated) formula for $\tau_3$, which we omit here because we will not require it in the present paper. The result of the computation for $\tau_2$ is that, in \emph{both} the Calabi-Yau and the nearly K\"ahler cases, $\tau_2 = 0$.
\end{proof}

\begin{remark}
The torsion forms for a $\G$-structure that is a warped product over a nearly K\"ahler $6$-manifold have previously appeared in Cleyton--Ivanov~\cite{Cleyton:Ivanov:2008}. The authors thank Sergey Grigorian for alerting them to this fact.
\end{remark}

The fact that these $\G$-structures always have vanishing two-torsion $\tau_2$ for any $h$ and $\theta$ will be useful later. Note that this is in stark contract to the closed $\G$-structures as studied in~\cite{ Bryant:Xu:2011, Bryant:2006, Weiss:Witt:2010, Xu:Ye:2010} where $\tau_2$ is the \emph{only} nonvanishing torsion form. It is for this reason that the sensible flow of $\G$-structures with such an $\SU{3}$ symmetry to consider this the Laplacian \emph{coflow} which we discuss in Section~\ref{coflowsec}.

\section{The Laplacian coflow of $\G$-structures} \label{coflowsec}

In this section we introduce the Laplacian \emph{coflow} of a coclosed $\G$-structure and discuss some of its general properties, including its soliton solutions. Then we concentrate specifically on the $\G$-structures~\eqref{phpseqs} arising from a warped product of $1$-manifold $L^1$ with a Calabi-Yau or a nearly K\"ahler $6$-manifold $N^6$.

\begin{definition} \label{coflowdefn}
We say that a time-dependent $\G$-structure \(\ph = \ph(t)\) on a \(7\)-manifold \(M^7\), defined for $t$ in some interval $[0, T)$, satisfies the \emph{Laplacian coflow equation} if for all
times \(t\) for which \(\ph(t)\) is defined, we have
\begin{equation} \label{cofloweq}
\pd{\ps}{t} = - \Delta_d \ps,
\end{equation}
where \(\ps(t)=\st{\ph(t)} \ph(t)\) is the Hodge dual $4$-form of $\ph(t)$ and \(\Delta_d = d d^* + d^* d\) is the Hodge Laplacian with respect to the metric $g(t) = g_{\ph(t)}$.
\end{definition}

In this paper, we will assume that this flow has short-time existence and uniqueness if we start with an initially coclosed $\G$-structure. This is very likely, since the flow is qualitatively very similar to the Laplacian flow $\pd{\ph}{t} = - \Delta_d \ph$ which does have short-time existence and uniqueness for an initially closed $\G$-structure. Also, entirely analogous to the fact that the Laplacian flow $\pd{\ph}{t} = - \Delta_d \ph$ preserves the closed condition, the Laplacian coflow $\pd{\ps}{t} = - \Delta_d \ps$ will preserve the coclosed condition. See~\cite{Bryant:Xu:2011, Bryant:2006, Xu:Ye:2010} for these results in the case of the Laplacian flow. The main goal of the present paper, in any case, is to study the soliton solutions to this flow in certain particular situations with symmetry.

\begin{remark} \label{coclosedrmk}
By equations~\eqref{torsionformseq}, a \(\G\)-structure is coclosed exactly when \(\tau_1= 0\) and \(\tau_2=0\).
\end{remark}

\subsection{Soliton solutions} \label{solitonsec}

As with the Ricci flow (and other geometric flows), it is of interest to consider ``self-similar solutions'' which are evolving by diffeomorphisms and scalings. If $f_t$ is a $1$-parameter family of diffeomorphisms generated by a vector field $X$ on $M$, and if $c(t) = 1 + \lambda t$, then differentiation shows that a coclosed $\G$-structure $\ph(t) = c(t) f_t^* \ph(t)$ is a solution to the coflow~\eqref{cofloweq} if and only if
\begin{equation} \label{coflowsolitoneq}
-\Delta_d \ps = \LieDer_{X} \ps + \lambda \ps \, = \, d(X \hook \ps) + \lambda \ps
\end{equation}
using the fact that $d\ps = 0$. In particular, a \emph{gradient coflow soliton} is a solution~\eqref{coflowsolitoneq} where $X = \nabla k$ for some $C^2$ function $k$ on $M$. As in the case of Ricci flow, we say that the soliton is \emph{expanding}, \emph{steady}, or \emph{shrinking} if $\lambda$ is positive, zero, or negative, respectively.
\begin{proposition} \label{compactsolitonsprop}
If $M^7$ is \emph{compact}, then there are no expanding or steady soliton solutions of~\eqref{coflowsolitoneq}, other than the trivial case of a torsion-free $\G$-structure in the steady case.
\end{proposition}
\begin{proof}
We take the wedge product of both sides of~\eqref{coflowsolitoneq} with $\ph = \st \ps$ and integrate over $M$ to obtain
\begin{equation} \label{nosolitonstempeq}
\int_M \langle \Delta_d \ps, \ps \rangle \, \vol{} + \lambda \int_M | \ps |^2 \, \vol{} + \int_M \langle d( X \hook \ps), \ps \rangle \, \vol{} \, = \, 0.
\end{equation}
Since $M$ is compact, we have
\begin{equation*}
\int_M \langle d( X \hook \ps), \ps \rangle \, \vol{} \, = \, \int_M  \langle X \hook \ps, d^* \ps \rangle \, \vol{}.
\end{equation*}
But the $\G$-structure is coclosed, so $\tau_1 = 0$ and hence $d^* \ps = \st{} d \st{} \ps = \st{} d \ph = \st{} (\tau_0 \ps + \st{} \tau_3) = \tau_0 \ph + \tau_3$. Therefore $d^* \ps$ lies in the space $\Lambda^4_1 \oplus \Lambda^4_{27}$, while $X \hook \ps$ lies in $\Lambda^4_7$. Since this decomposition of $\Lambda^4$ is pointwise orthogonal with respect to the metric $g_{\ph}$, we see that the last term in~\eqref{nosolitonstempeq} vanishes. Since $| \ps |^2 = 7$, we get
\begin{equation*}
\langle \langle \Delta_d \ps, \ps \rangle \rangle + 7 \lambda \int_M \vol{} \, = \, || d^* \ps ||^2 + 7 \lambda \, {\mathrm{Vol}}(M) \, = \, 0,
\end{equation*}
again using the fact that $d\ps = 0$. Thus we cannot have $\lambda >0$, and if $\lambda = 0$ then the $\G$-structure must be torsion-free. In the latter case $X$ must be a vector field generating a $\G$-symmetry: $\LieDer_{X} \ps = 0$. Since $M$ is compact, there will be no such nonzero $X$ unless $M$ has reducible holonomy (see~\cite{Joyce:2000}, for example.)
\end{proof}
\begin{remark} \label{shrinkersrmk}
It is easy to find nontrivial examples of compact shrinking solitons: a \emph{nearly} $\G$-structure is one for which $d\ps = 0$ and $d\ph = \mu \ps$ for some nonzero constant $\mu$. In this case $\Delta_d \ps = \mu^2 \ps$, and these give examples of compact shrinking solitions with $X = 0$ and $\lambda = - \mu^2$. Nearly $\G$ manifolds are those for which the metric cone over them has $\SP$ holonomy. There are many known compact examples. See~\cite{Bar:1993} or~\cite{Friedrich:et:al:1997} for more about nearly $\G$ manifolds.
\end{remark}
\begin{remark}
A very similar argument as in Proposition~\ref{compactsolitonsprop} can be used to show that for the Laplacian flow $\pd{\ph}{t} = - \Delta_d \ph$ of \emph{closed} $\G$-structures, in the compact case there are no expanding or steady solitons, other than the trivial case of a torsion-free $\G$-structure when $\lambda = 0$. The nearly $\G$ manifolds are examples of compact shrinking solitons for this flow as well.
\end{remark}

For the cohomogeneity one $\G$-structures that we consider in this paper, the only natural (with respect to the symmetry of the structure) vector fields must be of gradient type, so we will need only consider such gradient solitons, which are $C^2$ solutions $\ps(t)$ to
\begin{equation} \label{coflowgradientsolitoneq}
- \Delta_d \ps = \LieDer_{\nabla k} \ps + \lambda \ps
\end{equation}
for some $C^2$ function $k$ on $M$ and some constant $\lambda$. Also, for the examples we consider, $M^7 = N^6 \times L^1$, and while $N^6$ will always be taken to be compact, we can have either $L^1 \cong S^1$ or $L^1 \cong \R{}$, so we will not always be able to use Proposition~\ref{compactsolitonsprop}.

\subsection{The Hodge Laplacian on $M^7 = N^6 \times L^1$}

In this section we derive explicitly the Hodge Laplacian $-\Delta_d \ps$ for the $\G$-structures~\eqref{phpseqs} with $\SU{3}$ symmetry when $N^6$ is Calabi-Yau or nearly K\"ahler. Recall that we consider only coclosed $\G$-structures of these types. By Lemma~\ref{torsionlemma}, the two-torsion $\tau_2$ is always zero, but we need to \emph{impose} the condition that $\tau_1 = 0$, which, as we noted above, will be preserved under the Laplacian coflow $\pd{\ps}{t} = - \Delta_d \ps$.
\begin{assumption} \label{tauonevanishassume}
The $\G$-structure~\eqref{phpseqs} is assumed to be coclosed. Thus $\tau_1 = 0$. By Lemma~\ref{torsionlemma}, this means that we assume:
\begin{equation} \label{assumeeq}
\left.
\begin{aligned}
& \text{when $N^6$ is Calabi-Yau:} && h' = 0; \\ & \text{when $N^6$ is nearly K\"ahler: } && h' = G \cos 3 \theta.
\end{aligned}
\quad \quad \right\}
\end{equation}
\end{assumption}
With this assumption, it is easy to compute $- \Delta_d \ps = - dd^* \ps = - d \st{7} \! d \ph$.
\begin{lemma} \label{symmetrylaplacianlemma}
For these $\G$-structures, we have that
\begin{equation} \label{symmetrylaplacianeq}
\left.
\begin{aligned}
& \text{when $N^6$ is Calabi-Yau:} - \Delta_d \ps =  \left( \frac{i(F^3)'}{2G} \right)' dr \wedge \Omega + \left( - \frac{i(\bar F^3)'}{2G} \right)' dr \wedge \bar \Omega; \\ & \text{when $N^6$ is nearly K\"ahler: } - \Delta_d \ps = A dr \wedge \Omega + \bar A dr \wedge \bar \Omega + B \frac{\omega^2}{2}, \\ & \qquad \text{where } \, A = \left[ \left( \frac{i(F^3)'}{2G} - \frac{3i}{2} h^2 \right)' + 6 G h \sin 3 \theta \right] \, \text { and } \, B = \left[ -\frac{4}{G} (h^3 \cos 3 \theta)' + 12 h^2 \right].
\end{aligned}
\quad \quad \right\}
\end{equation}
\end{lemma}
\begin{proof}
We use the expression for $\st{7} \! (d\ph)$ that we derived in~\eqref{starfirstderivativeeqs}. In the Calabi-Yau case, we have
\begin{equation*}
- \Delta_d \ps \, = \, - d \st{7} \! (d\ph) \, = \, \left( \frac{i(F^3)'}{2G} \right)' dr \wedge \Omega + \left( - \frac{i(\bar F^3)'}{2G} \right)' dr \wedge \bar \Omega
\end{equation*}
since $d\Omega = 0$ and $d\bar \Omega = 0$. This establishes the first half of~\eqref{symmetrylaplacianeq}. In the nearly K\"ahler case, we use also~\eqref{NKequationscomplex} to obtain
\begin{align*}
- \Delta_d \ps \, = \, - d \st{7} \! (d\ph) & = \, \left( \frac{i(F^3)'}{2G} - \frac{3i}{2} h^2 \right)' dr \wedge \Omega + \left( \frac{-i(F^3)'}{2G} + \frac{3i}{2} h^2 \right)' dr \wedge \bar \Omega \\ & \quad {} + \left(\frac{i(F^3)'}{2G} - \frac{3i}{2} h^2 \right) d\Omega +  \left( \frac{-i(\bar F^3)'}{2G} + \frac{3i}{2} h^2 \right) d\bar \Omega + \frac{2 i G}{h^2} (F^3 - \bar F^3) dr \wedge d\omega \\
& = \left[ \left( \frac{i(F^3)'}{2G} - \frac{3i}{2} h^2 \right)' - \frac{3 i G}{h^2} (F^3 - \bar F^3) \right] dr \wedge \Omega \\ & \quad {} + \left[ \left( \frac{-i(F^3)'}{2G} + \frac{3i}{2} h^2 \right)' - \frac{3 i G}{h^2} (F^3 - \bar F^3) \right] dr \wedge \bar \Omega \\  & \quad {} + \left[ 4i \left(\frac{i(F^3)'}{2G} - \frac{3i}{2} h^2 \right) -4i \left(\frac{-i(F^3)'}{2G} + \frac{3i}{2} h^2 \right) \right] \frac{\omega^2}{2}.
\end{align*}
Using $F = h e^{i \theta}$, this expression simplifies to
\begin{align*}
- \Delta_d \ps & = \, \left[ \left( \frac{i(F^3)'}{2G} - \frac{3i}{2} h^2 \right)' + 6 G h \sin 3 \theta \right] dr \wedge \Omega \\ & \quad {} + \left[ \left( \frac{-i(F^3)'}{2G} + \frac{3i}{2} h^2 \right)' + 6 G h \sin 3 \theta \right] dr \wedge \bar \Omega \\  & \quad {} + \left[ -\frac{4}{G} (h^3 \cos 3 \theta)' + 12 h^2 \right] \frac{\omega^2}{2}.
\end{align*}
which establishes the second half of~\eqref{symmetrylaplacianeq}.
\end{proof}
Recall that we have
\begin{equation} \label{pstempeq}
\ps = \frac{i G F^3}{2} dr \wedge \Omega - \frac{i G \bar F^3}{2} dr \wedge \bar \Omega - h^4 \frac{\omega^2}{2}.
\end{equation}
In the Laplacian coflow $\pd{\ps}{t} = - \Delta_d \ps$, only the functions $F = h e^{i\theta}$ and $G$ depend on $t$ and the  coordinate $r$ on $L^1$. We are now ready to study the coflow and corresponding soliton equations in detail for the Calabi-Yau and the nearly K\"ahler cases in the next two sections.

\section{The case when $N^6$ is Calabi-Yau} \label{CYsec}

We begin with the evolution equations.

\subsection{The CY evolution equations} \label{CYevolutionsec}

\begin{theorem} \label{CYcoflowthm}
Let $N^6$ be Calabi-Yau, and let $M^7 = N^6 \times L^1$ be a manifold with \emph{coclosed} $\G$-structure given by~\eqref{phpseqs}, with $d\ps = 0$. Then under the Laplacian coflow $\pd{\ps}{t} = - \Delta_d \ps$, the functions $F = h e^{i \theta}$ and $G$ on $L^1$ (depending also on the  time parameter $t$) satisfy the following evolution equations.
\begin{align} \label{CYfloweqs1}
& h \, = \, 1, & &  \pd{\theta}{t} \, = \, \Delta \theta, & & \pd{G}{t} \, = \, - 9 G | \nabla \theta |^2,
\end{align}
where the rough Laplacian $\Delta$, the gradient $\nabla$, and the pointwise norm $| \cdot |$ are all taken with respect to the metric $g_7 = G^2 dr^2 + h^6 g_6$ on $M^7$.
\end{theorem}
\begin{proof}
Differentiating~\eqref{pstempeq} with respect to $t$ and using Lemma~\ref{symmetrylaplacianlemma}, we can compute $\pd{\ps}{t} = - \Delta_d \ps$ and equate the coefficients of $dr \wedge \Omega$, $dr \wedge \bar \Omega$, and $\frac{\omega^2}{2}$. We find that
\begin{equation*}
\frac{\partial}{\partial t} \left( \frac{i G F^3}{2} \right) \, = \, \left( \frac{i(F^3)'}{2G} \right)' \qquad \text{ and } \qquad \frac{\partial}{\partial t} (- h^4) \, = \, 0.
\end{equation*}
The second equation says that $\pd{h}{t} = 0$, so that $h$ is constant in time as well. (Recall from~\eqref{assumeeq} that the condition $\tau_1 = 0$ in this case was that $h$ is also independent of $r$.) Without loss of generality, by rescaling the metric on the Calabi-Yau manifold $N^6$, we can assume that $h=1$ from now on. Substituting $h=1$ into the first equation above and simplifying, we obtain
\begin{equation*}
\frac{\partial}{\partial t} \left( G e^{i 3 \theta} \right) \, = \, \left( \frac{(e^{i 3 \theta})'}{G} \right)'.
\end{equation*}
Expanding and simplifying, we have
\begin{equation*}
\left( \pd{G}{t} + 3 i G \pd{\theta}{t} \right) e^{i 3 \theta} \, = \, \left( \frac{3i \theta'}{G} e^{i3 \theta} \right)' \, = \, \left( - \frac{3i G' \theta'}{G^2} + \frac{3i \theta''}{G} - \frac{9 (\theta')^2}{G} \right) e^{3i \theta}.
\end{equation*}
Equating real and imaginary parts gives
\begin{equation*}
\pd{G}{t} \, = \, - \frac{9(\theta')^2}{G}, \qquad \qquad \pd{\theta}{t} \, = \, \frac{\theta''}{G^2} - \frac{G' \theta'}{G^3}. 
\end{equation*}
Since in this case we have $h=1$, equations~\eqref{roughlaplacianeq} and~\eqref{gradienteq} give that the above equations can be invariantly expressed as
\begin{equation*}
\pd{G}{t} \, = \, - 9 G | \nabla \theta |^2, \qquad \qquad \pd{\theta}{t} \, = \, \Delta \theta, 
\end{equation*}
which is what we wanted to prove.
\end{proof}

Note that even though the phase function $\theta(r,t)$ satisfies what appears to be a simple heat equation, the Laplacian $\Delta$ is taken with respect to the metric~\eqref{symmetrymetriceq} that is changing with time. This makes it very difficult to establish long-time existence without a much more delicate analysis. In general, we expect that there should be singularity formation in finite time, as is the case with most geometric evolution equations.

\subsection{The CY soliton equations} \label{CYsolitonsec}

Next, we turn to the \emph{soliton} equations in this case. Since this is a time-static situation, we can without loss of generality reparametrize the local coordinate $r$ so that $G = 1$, as discussed in Remark~\ref{Gremark}. We are looking for soliton solutions which have the same $\SU{3}$ symmetry as the evolution equations,  so the only possible vector fields are of the form $X = s(r) \frac{\partial}{\partial r}$ for some function $s = s(r)$ on $L^1$. By letting $k(r) = \int_{r_o}^r s(u) du$ be an antiderivative, we can assume that $X = \nabla k = k' \frac{\partial}{\partial r}$ is a gradient vector field for some function $k = k(r)$ on $L^1$. Note that since $G = 1$, we do indeed have $(dr)^{\sharp} = \frac{\partial}{\partial r}$ so this is the correct expression for $\nabla k$. The soliton equation, as derived in~\eqref{coflowgradientsolitoneq}, is
\begin{equation} \label{solitonCYtempeq}
- \Delta_d \ps = \LieDer_{\nabla k} \ps + \lambda \ps = d({\nabla k} \hook \ps) + \lambda \ps
\end{equation}
since $d\ps = 0$.
\begin{theorem} \label{CYsolitonthm}
The coclosed $\G$-structures which satisfy the soliton equation~\eqref{solitonCYtempeq} when $N^6$ is Calabi-Yau are given by~\eqref{phpseqs} where $G=1$ and
\begin{equation*}
\lambda = 0, \qquad \quad h = 1, \qquad \quad \theta = \frac{2}{3} \arctan ( c e^{br} ), \qquad \quad X = b \left( \frac{1 - c^2 r^{2br}}{1 + c^2 e^{2br}} \right) \frac{\partial}{\partial r},
\end{equation*}
for some real constants $b$ and $c$. In particular, all the soliton solutions are \emph{steady} and the only solutions which exist in the case that $L^1 \cong S^1$ is \emph{compact} are constant $\theta$ and $k'$ (corresponding to $b = 0$ or $c=0$) which are trivial translations and phase rotations of the standard torsion-free $\G$-structure on $N^6 \times S^1$. However, in the case where $L^1 \cong \R{}$ is \emph{noncompact}, we do obtain nontrivial soliton solutions on $N^6 \times \R{}$.
\end{theorem}
\begin{proof}
We compute using~\eqref{pstempeq} and $G=1$ that
\begin{equation*}
d({\nabla k} \hook \ps) = d \left( k' \frac{\partial}{\partial r} \hook \ps \right) \, = \, d \left( \frac{i F^3 k'}{2}  \Omega - \frac{i \bar F^3 k'}{2} \bar \Omega \right) \, = \, \frac{i}{2} (F^3 k')' dr \wedge \Omega - \frac{i}{2} (\bar F^3 k')' dr \wedge \bar \Omega.
\end{equation*}
Substituting the above expression into~\eqref{solitonCYtempeq} and using~\eqref{pstempeq} and~\eqref{symmetrylaplacianeq}, and comparing coefficients, we have
\begin{equation*}
\frac{i}{2} (F^3)'' \, = \, \lambda \frac{i F^3}{2} + \frac{i}{2} (F^3 k')', \qquad \qquad 0 \, = \, - \lambda h^4.
\end{equation*}
Since $h > 0$, the second equation says $\lambda = 0$. That is, there are \emph{only} steady solitons in this case. Comparing with Remark~\ref{shrinkersrmk}, this implies  (at least if $L^1$ is compact) that this $\G$-structure cannot be nearly $\G$. Indeed, it is easy to check directly that for this ansatz, the three-torsion $\tau_3$ will vanish only when $\tau_0$ also vanishes and $\ph$ is completely torsion-free. Now with $\lambda = 0$, and recalling that $h=1$, the first equation above simplifies to
\begin{equation*}
(e^{i 3 \theta})'' - (e^{i 3 \theta} k')' \, = \, 0,
\end{equation*}
which can be immediately integrated once to yield
\begin{equation*}
(e^{i 3 \theta})' - e^{i 3 \theta} k' \, = \, -b = -(b_1 + i b_2)
\end{equation*}
for some constant $b \in \C{}$. Taking real and imaginary parts, we get
\begin{equation} \label{solitonCYtempeq2}
(\cos 3 \theta)' - (\cos 3 \theta) k' = -b_1, \qquad \qquad (\sin 3 \theta)' - (\sin 3 \theta) k' = -b_2.
\end{equation}
In~\eqref{solitonCYtempeq2}, if we multiply the first equation by $\sin 3 \theta$ and the second equation by $\cos 3 \theta$ and subtract, we eliminate $k'$ and obtain
\begin{align*}
-b_1 \sin 3 \theta + b_2 \cos 3 \theta & = (\sin 3 \theta) (\cos 3 \theta)' - (\cos 3 \theta) (\sin 3 \theta)' \\
& = - 3 \theta' \sin^2 3 \theta - 3 \theta' \cos^2 3 \theta = - 3 \theta'
\end{align*}
and thus
\begin{equation} \label{solitonCYthetaeq}
3\theta' = b_1 \sin 3\theta - b_2 \cos 3 \theta.
\end{equation}
But in~\eqref{solitonCYtempeq2}, we can also multiply the first equation by $\cos 3 \theta$ and the second equation by $\sin 3 \theta$ and add, and we find that
\begin{equation*}
-b_1 \cos 3 \theta - b_2 \sin 3 \theta = (\cos 3\theta) (\cos 3\theta)' + (\sin 3 \theta) (\sin 3 \theta)' - (\cos^2 3 \theta) k' - (\sin^2 3 \theta) k' = - k'
\end{equation*}
and therefore
\begin{equation} \label{solitonCYkeq}
k' = b_1 \cos 3\theta + b_2 \sin 3 \theta.
\end{equation}
Equation~\eqref{solitonCYthetaeq} can actually be integrated exactly, although the solution is quite complicated for general $b \in \C{}$. However, given any values $\theta(r_o)$ and $k'(r_o)$ of the functions $\theta$ and $k'$ at some fixed point $r_o \in L^1$, we see that by performing a ``rotation'' of the Calabi-Yau holomorphic volume form $\Omega \mapsto e^{i \gamma} \Omega$ for an appropriate constant $\gamma$, we can arrange that $b_2 = 0$ so $b = b_1$ is purely real. We are always free to do such a rotation because the holomorphic volume form $\Omega$ of a Calabi-Yau manifold is only defined up to a constant phase factor. Then equation~\eqref{solitonCYthetaeq}  becomes
\begin{equation*}
\frac{3 d\theta}{\sin 3\theta} \, = \, b \, dr
\end{equation*}
which has solution
\begin{equation*}
\theta(r) \, = \, \frac{2}{3} \arctan ( c e^{br} )
\end{equation*}
for some real constants $b$ and $c$ depending on the ``initial'' conditions. This can then be substituted into~\eqref{solitonCYkeq} to directly solve for $k'$. We have
\begin{equation*}
k' \, = \, b \cos (2 \arctan (c e^{br})) \, = \, b \left( \frac{1 - c^2 r^{2br}}{1 + c^2 e^{2br}} \right),
\end{equation*}
and the proof is complete.
\end{proof}

We remark that since $h=1$ and $G=1$ for these soliton solutions, the metric~\eqref{symmetrymetriceq} on $M^7$ is just the product of the flat metric on $L^1$ and the Calabi-Yau metric on $N^6$. While the metric in this case is not new, the corresponding $\G$-structure $\ph$ is in general \emph{not} torsion-free. Indeed, from Lemma~\ref{torsionlemma} we see that $\tau_0$ will not vanish unless $\theta$ is constant. This is similar, but slightly different, to the fact that the standard Euclidean metric on $\R{n}$ can be written in a non-trivial way as a gradient Ricci soliton.

\section{The case when $N^6$ is nearly K\"ahler} \label{NKsec}

Now suppose that $N^6$ is nearly K\"ahler. Again we begin with the evolution equations.

\subsection{The NK evolution equations} \label{NKevolutionsec}

\begin{theorem} \label{NKcoflowthm}
Let $N^6$ be nearly K\"ahler, and let $M^7 = N^6 \times L^1$ be a manifold with \emph{coclosed} $\G$-structure given by~\eqref{phpseqs}, with $d\ps = 0$. Then under the Laplacian coflow $\pd{\ps}{t} = - \Delta_d \ps$, the functions $F = h e^{i \theta}$ and $G$ on $L^1$ (depending also on the  time parameter $t$) satisfy the following evolution equations.
\begin{align} \label{NKcofloweqs1}
& \frac{\partial h}{\partial t} = \Delta h - \frac{3}{h}\left(1 +{ |\nabla h|}^2 \right), & & \pd{\theta}{t} = \Delta \theta - \frac{\sin 6 \theta}{h^2}, & & \pd{G}{t} = - \left(9 \, {|\nabla \theta|}^2 + 3 \, {\left| \frac{\sin 3 \theta}{h} \right|}^2 \right) G,
\end{align}
where the rough Laplacian $\Delta$, the gradient $\nabla$, and the pointwise norm $| \cdot |$ are all taken with respect to the metric $g_7 = G^2 dr^2 + h^6 g_6$ on $M^7$.
\end{theorem}
\begin{proof}
Differentiating~\eqref{pstempeq} with respect to $t$ and using Lemma~\ref{symmetrylaplacianlemma}, we can compute $\pd{\ps}{t} = - \Delta_d \ps$ and equate the coefficients of $dr \wedge \Omega$, $dr \wedge \bar \Omega$, and $\frac{\omega^2}{2}$. We find that
\begin{equation} \label{NKcoflowtempeqa}
\frac{\partial}{\partial t} \left( \frac{i G F^3}{2} \right) \, = \, \left( \frac{i(F^3)'}{2G} - \frac{3i}{2} h^2 \right)' + 6 G h \sin 3 \theta \quad \text{ and } \quad \frac{\partial}{\partial t} (- h^4) \, = \, -\frac{4}{G} (h^3 \cos 3 \theta)' + 12 h^2.
\end{equation}
The first equation is a complex equation, and can be simplified to
\begin{equation} \label{NKcoflowtempeq}
\frac{\partial}{\partial t} \left( G F^3 \right) \, = \, \left( \frac{(F^3)'}{G} - 3 h^2 \right)' - 12 i G h \sin 3 \theta.
\end{equation}
The second equation is a real equation and can be simplified to
\begin{equation} \label{NKcoflowtempeq2}
\frac{\partial h}{\partial t} \, = \, \frac{(h^3 \cos 3 \theta)'}{G h^3} - \frac{3}{h}.
\end{equation}
Recall, however, that we \emph{also} have the $\tau_1 = 0$ condition from~\eqref{assumeeq} that says
\begin{equation} \label{NKcoflowtempeq3}
h' \, = \, G \cos 3 \theta.
\end{equation}
Now at first glance it would appear that this system is overdetermined, because we have four equations for three functions $G$, $h$, and $\theta$. However, we will now see that there is indeed some redundancy. The real part of~\eqref{NKcoflowtempeq} is
\begin{equation*}
\frac{\partial}{\partial t} \left( G h^3 \cos 3 \theta \right) \, = \, \left[ \frac{(h^3 \cos 3 \theta)'}{G} - 3 h^2 \right]'.
\end{equation*}
If we substitute~\eqref{NKcoflowtempeq3} into the left hand side of the above expression, we obtain
\begin{equation*}
\frac{\partial}{\partial t} \left( h^3 h' \right) \, = \, \frac{\partial}{\partial t} \left( \frac{h^4}{4} \right)' \, = \, \left[ \frac{(h^3 \cos 3 \theta)'}{G} - 3 h^2 \right]'
\end{equation*}
which, up to a factor of $(-4)$, is exactly the derivative with respect to $r$ of the second equation in~\eqref{NKcoflowtempeqa} which led to~\eqref{NKcoflowtempeq2}. Thus, the independent equations are~\eqref{NKcoflowtempeq2} and~\eqref{NKcoflowtempeq3} and the imaginary part of~\eqref{NKcoflowtempeq}:
\begin{equation} \label{NKcoflowtempeq4}
\frac{\partial}{\partial t} \left( G h^3 \sin 3 \theta \right) \, = \, \left( \frac{(h^3 \sin 3 \theta)'}{G} \right)' - 12 Gh \sin 3 \theta.
\end{equation}
We need to extract invariant expressions for the time derivatives of $G$, $h$, and $\theta$ from the above equations. We begin by substituting~\eqref{NKcoflowtempeq3} into~\eqref{NKcoflowtempeq2} to eliminate $\cos 3 \theta$:
\begin{align} \nonumber
 \frac{\partial h}{\partial t} & = \, \frac{1}{G h^3} \left(\frac{h^3 h'}{G}\right)' - \frac{3}{h} \, = \, \frac{1}{Gh^3} \left[ \frac{3 h^2 (h')^2}{G} + \frac{h^3h''}{G} - \frac{h^3 h' G'}{G^2} \right] - \frac{3}{h} \\  \label{NKcoflowtempheq} \frac{\partial h}{\partial t} & = \, \frac{h''}{G^2} + \frac{3 (h')^2}{h G^2} - \frac{h' G'}{G^3} - \frac{3}{h}.
 \end{align}
The above form of $\frac{\partial h}{\partial t}$ will be useful later. We can further simplify it as
 \begin{equation*}
\frac{\partial h}{\partial t}  \, = \, \left( \frac{h''}{G^2} + \frac{6 (h')^2}{h G^2} - \frac{h' G'}{G^3}\right) - \frac{3}{h} - \frac{3 (h')^2}{h G^2} \, = \, \Delta h - \frac{3}{h}\left(1 +{ |\nabla h|}^2 \right),
\end{equation*}
where we have used~\eqref{roughlaplacianeq} and~\eqref{gradienteq}. This proves the first part of~\eqref{NKcofloweqs1}. We will need to work a bit harder to get the evolution equations for $\theta$ and $G$. Let $S$ denote the right hand side of~\eqref{NKcoflowtempeq4}:
\begin{equation} \label{Seq}
S \, = \, \left( \frac{(h^3 \sin 3 \theta)'}{G} \right)' - 12 Gh \sin 3 \theta.
\end{equation}
Expanding the left hand side of~\eqref{NKcoflowtempeq4} and rearranging, we find
\begin{equation} \label{NKcoflowtempeq5}
(h^3 \sin 3 \theta) \pd{G}{t} + (3 G h^3 \cos 3 \theta) \pd{\theta}{t} \, = \, S - (3 G h^2 \sin 3 \theta) \pd{h}{t}.
\end{equation}
This equation is linear in $\pd{G}{t}$ and $\pd{\theta}{t}$. We can get another one by differentiating~\eqref{NKcoflowtempeq3} with respect to $t$:
\begin{equation*}
(\cos 3 \theta) \pd{G}{t} - (3 G \sin 3 \theta) \pd{\theta}{t} \, = \, \pd{}{t} h' \, = \, \left( \pd{h}{t} \right)'.
\end{equation*}
Dividing~\eqref{NKcoflowtempeq5} by $h^3$, we now have the following system of linear equations:
\begin{equation*}
\begin{pmatrix} \cos 3 \theta & \sin 3 \theta \\ & \\ - \sin 3 \theta & \cos 3 \theta \end{pmatrix} \begin{pmatrix} 3 G \pd{\theta}{t} \\ \\ \pd{G}{t} \end{pmatrix} \, = \, \begin{pmatrix} \frac{S}{h^3} - \frac{3G \sin 3 \theta}{h} \pd{h}{t} \\ \\ \left( \pd{h}{t} \right)' \end{pmatrix}.
\end{equation*}
This system is easily solved to yield
\begin{equation*}
\begin{aligned}
3 G \pd{\theta}{t} & = \, (\cos 3 \theta) \left( \frac{S}{h^3} - \frac{3G \sin 3\theta}{h} \pd{h}{t} \right) - \sin 3 \theta \left( \pd{h}{t} \right)', \\
\pd{G}{t} & = \, (\sin 3 \theta) \left( \frac{S}{h^3} - \frac{3G \sin 3\theta}{h} \pd{h}{t} \right) + \cos 3 \theta   \left( \pd{h}{t} \right)'.
\end{aligned}
\end{equation*}
We can now substitute the expression~\eqref{Seq} for $S$, the expression~\eqref{NKcoflowtempheq} for $\pd{h}{t}$, and the derivative with respect to $r$ of~\eqref{NKcoflowtempheq} for $(\pd{h}{t})'$ into the above equations. We also repeatedly use~\eqref{NKcoflowtempeq3} to eliminate all terms involving $h'$ at every stage. After much computation, the  result is:
\begin{align} \label{NKcoflowtempeq6}
\pd{G}{t} & = \, - \frac{3G \sin^2 3 \theta}{h^2} - \frac{9 (\theta')^2}{G}, \\ \label{NKcoflowtempeq7}
\pd{\theta}{t} & = \, \frac{\theta''}{G^2} + \frac{6 \theta' \cos 3 \theta}{hG} - \frac{\theta' G'}{G^3} - \frac{2 \sin 3 \theta \cos 3 \theta}{h^2}.
\end{align}
Now~\eqref{gradienteq} shows that~\eqref{NKcoflowtempeq6} becomes
\begin{equation*}
\pd{G}{t} \, = \, - \left(9 \, {|\nabla \theta|}^2 + 3 \, {\left| \frac{\sin 3 \theta}{h} \right|}^2 \right) G,
\end{equation*}
which is the third part of~\eqref{NKcofloweqs1}. Finally, substituting $\cos 3 \theta = \frac{h'}{G}$ in~\eqref{NKcoflowtempeq7} and using~\eqref{roughlaplacianeq} gives
\begin{equation*}
\pd{\theta}{t} \, = \, \Delta \theta - \frac{\sin 6 \theta}{h^2},
\end{equation*}
which is the second part of~\eqref{NKcofloweqs1}.
\end{proof}

As discussed at the end of Section~\ref{CYsolitonsec}, long-time existence for these evolution equations would be difficult to determine, and in general one should expect singularity formation in finite time. 

\subsection{The NK soliton equations} \label{NKsolitonsec}

Now we turn to the \emph{soliton} equations in the nearly K\"ahler case. As before, we can without loss of generality reparametrize the local coordinate $r$ so that $G = 1$. Also as in the Calabi-Yau case, we can assume that $X = \nabla k = k' \frac{\partial}{\partial r}$ is a gradient vector field for some function $k = k(r)$ on $L^1$. We recall again that the soliton equation, as derived in~\eqref{coflowgradientsolitoneq}, is
\begin{equation} \label{solitonNKtempeq}
- \Delta_d \ps = \LieDer_{\nabla k} \ps + \lambda \ps = d({\nabla k} \hook \ps) + \lambda \ps
\end{equation}
since $d\ps = 0$.
\begin{theorem} \label{NKsolitonthm}
The coclosed $\G$-structures which satisfy the soliton equation~\eqref{solitonNKtempeq} when $N^6$ is nearly K\"ahler are given by~\eqref{phpseqs} where $G=1$ and the functions $h$, $\theta$, and $k'$  satisfy
\begin{align} \label{NKsolitoneq1}
h' & = \, \cos 3 \theta, \\ \label{NKsolitoneq2} 0 & = \, (h^3 \sin 3 \theta)'' - 12h \sin 3 \theta - \lambda h^3 \sin 3 \theta - (k' h^3 \sin 3 \theta)', \\ \label{NKsolitoneq3} 0 & = \, (h^3 \cos 3 \theta)' -3 h^2 - \frac{\lambda}{4} h^4 - k' h^3 \cos 3 \theta.
\end{align}
\end{theorem}
\begin{proof}
As in the proof of Theorem~\ref{CYsolitonthm}, but this time using~\eqref{NKequationscomplex}, we find that
\begin{align*}
d({\nabla k} \hook \ps) & = \, d \left( k' \frac{\partial}{\partial r} \hook \ps \right) \, = \, d \left( \frac{i F^3 k'}{2}  \Omega - \frac{i \bar F^3 k'}{2} \bar \Omega \right) \\ & = \, \frac{i}{2} (F^3 k')' dr \wedge \Omega - \frac{i}{2} (\bar F^3 k')' dr \wedge \bar \Omega + \frac{i F^3 k'}{2} d\Omega - \frac{i \bar F^3 k'}{2} d\bar \Omega \\ & = \, \frac{i}{2} (F^3 k')' dr \wedge \Omega - \frac{i}{2} (\bar F^3 k')' dr \wedge \bar \Omega -2 (F^3 + \bar F^3) k' \frac{\omega^2}{2}.
\end{align*}
We substitute the above expression into~\eqref{solitonNKtempeq} and use $G=1$ and equations~\eqref{pstempeq} and~\eqref{symmetrylaplacianeq}. When we compare coefficients, we find that
\begin{equation*}
 \left( \frac{i(F^3)'}{2} - \frac{3i}{2} h^2 \right)' + 6 h \sin 3 \theta\, = \, \lambda \frac{i F^3}{2} + \frac{i}{2} (F^3 k')', \qquad \qquad -4 (h^3 \cos 3 \theta)' + 12 h^2 \, = \, - \lambda h^4 - 2(F^3 + \bar F^3) k'.
\end{equation*}
Taking real and imaginary parts of the first equation, and simplifying all three equations, we get
\begin{align} \label{NKsolitontempeq1}
\left( (h^3 \cos 3 \theta) - 3 h^2 \right)' - \lambda h^3 \cos 3 \theta - (k' h^3 \cos 3 \theta)' \, = \, 0, \\ \label{NKsolitontempeq2} (h^3 \sin 3 \theta)'' - 12h \sin 3 \theta - \lambda h^3 \sin 3 \theta - (k' h^3 \sin 3 \theta)' \, = \, 0, \\ \label{NKsolitontempeq3} (h^3 \cos 3 \theta)' -3 h^2 - \frac{\lambda}{4} h^4 - k' h^3 \cos 3 \theta \, = \, 0.
\end{align}
As in Theorem~\ref{NKcoflowthm}, this appears to be overdetermined because we also have the $\tau_1 = 0$ assumption~\eqref{assumeeq} which is now $\cos 3 \theta = h'$, but it is easy to see that with this condition, equation~\eqref{NKsolitontempeq1} is a consequence of equation~\eqref{NKsolitontempeq3}. This completes the proof.
\end{proof}

We now attempt to solve the system of equations in Theorem~\ref{NKsolitonthm}. It is easy to spot some particular solutions. For example, if we assume $3\theta = 0$, then~\eqref{NKsolitoneq2} is trivially satisfied and~\eqref{NKsolitoneq1} implies that $h = r + b$ for some constant $b$. Then~\eqref{NKsolitoneq3} becomes $\lambda (r + b) + 4k' = 0$, which shows that we can find a $k'$ for any choice of $\lambda$. Thus one family of solutions is:
\begin{equation*}
3\theta = 0, \qquad h = r + b, \qquad k' = - \frac{\lambda}{4} (r + b), \qquad \lambda, b \in \R{}.
\end{equation*}
Similarly, if we assume $3\theta = \pi$, then we find the following family of solutions:
\begin{equation*}
3\theta = \pi, \qquad h = -r + b, \qquad k' = \frac{\lambda}{4} (-r + b), \qquad \lambda, b \in \R{}.
\end{equation*}
Since we must have $h >0$ always, we see that the above two families of solutions are only defined on some proper subinterval of $L^1 = \R{1}$. In particular, these families include the case of the Riemannian cone over $N^6$, given by $h(r) = r$ with $L^1 = (0, \infty)$. The $\G$-structure $\ph$ is torsion-free in this case and $M^7$ has $\G$ holonomy. This example is entirely analogous to the exhibition of the standard Euclidean metric on $\R{n}$ as a non-trivial gradient Ricci soliton.

Another family of special solutions can be found if we assume $3\theta = \pm \frac{\pi}{2}$. In this case~\eqref{NKsolitoneq1} implies that $h = b$ for some constant $b > 0$, and then~\eqref{NKsolitoneq3} forces $\lambda = -\frac{12}{b^2}$ and~\eqref{NKsolitoneq2} then gives $k'' = 0$. Thus another family of solutions is:
\begin{equation*}
\theta = \frac{\pi}{2}, \qquad h = b, \qquad k' = c, \qquad \lambda = - \frac{12}{b^2}, \qquad b > 0, c \in \R{}.
\end{equation*}
Notice that this family of solutions are all \emph{shrinkers}. In this case the metric~\eqref{symmetrymetriceq} on $M^7$ is a Riemannian product. 

Finally, we can find a more interesting solution by trying $h(r) = \sin(r)$. The motivation for such an ansatz is that ``sine-cone'' metrics $g_M = dr^2 + \sin^2(r) g_N$ arise often in the study of Einstein manifolds (see for example~\cite{Boyer:Galicki:2008} or~\cite{Fernandez:et:al:2008}) and the fact that $h' = \cos(3 \theta)$. One can check that this ansatz does indeed work and we obtain the following solution:
\begin{equation*}
3\theta =r, \qquad h = \sin(r), \qquad k' = 0, \qquad \lambda = -16.
\end{equation*}
This is another shrinking soliton. In this case, $L^1 = (0, \pi)$ and the manifold $M^7 = (0, \pi) \times N^6$ can be compactified to a compact topological space with two ``conical singularities.'' One can also check (for example using the formulae on page 192 of~\cite{Sternberg}) that in this case, the metric $g_M$ on $M$ is Einstein. This $\G$-structure is \emph{not} torsion-free, but by equation~\eqref{solitonNKtempeq} the $3$-form $\ph$ is an eigenform (with eigenvalue $16$) of its induced Hodge Laplacian $\Delta_d$.

In the general case, we can reduce the equations of Theorem~\ref{NKsolitonthm} to a single \emph{third order} nonlinear ordinary differential equation for $h$ as follows. Let us assume that $h' = \cos 3 \theta$ is never zero. We know that $h = r + b$ and $\theta = 0$ is a solution with this property, so we are looking for other solutions close to this one. First, we substitute~\eqref{NKsolitoneq1} into~\eqref{NKsolitoneq3} to obtain
\begin{align*}
0 & = \, (h^3 h')' - 3 h^2 - \frac{\lambda}{4} h^4 - k' h^3 h' \\ & = \, 3 h^2 (h')^2 + h^3 h'' - 3 h^2 - \frac{\lambda}{4} h^4 - h^3 h' k'.
\end{align*}
We can solve the above expression for $h^3 k'$ as:
\begin{equation} \label{generalsolitontempeq1}
h^3 k' \, = \, 3 h^2 h' + \frac{h^3 h''}{h'} - \frac{3 h^2}{h'} - \frac{\lambda h^4}{4 h'}.
\end{equation}
We will also need the derivative of the above expression:
\begin{align} \nonumber
(h^3 k')' & = \, (6 h (h')^2 + 3 h^2 h'') + \left( 3 h^2 h'' + \frac{h^3 h'''}{h'} - \frac{h^3 (h'')^2}{(h')^2} \right) + \left( - 6 h + \frac{3 h^2 h''}{(h')^2} \right) + \left( -\lambda h^3 + \frac{\lambda h^4 h''}{4 (h')^2} \right) \\ \label{generalsolitontempeq2} & = \, 6 h (h')^2 + 6 h^2 h'' - 6 h - \lambda h^3 + \frac{h^3 h'''}{h'} + \frac{3 h^2 h''}{(h')^2} + \frac{\lambda h^4 h''}{4 (h')^2} - \frac{h^3 (h'')^2}{(h')^2}.
\end{align}
Let us write $u = \sin 3 \theta$ to simplify notation. Then equation~\eqref{NKsolitoneq2} is
\begin{align} 
0 & = \, (h^3 u)'' - 12h u - \lambda h^3 u - (h^3 k' u)' \\ \label{generalsolitontempeq3} & = \, 6 h (h')^2 u + 3 h^2 h'' u + 6 h^2 h' u' + h^3 u'' - 12h u - \lambda h^3 u - (h^3 k')' u - (h^3 k') u'
\end{align}
We can substitute~\eqref{generalsolitontempeq1} and~\eqref{generalsolitontempeq2} into~\eqref{generalsolitontempeq3} to completely eliminate $k'$. After some simplification, the end result is
\begin{equation} \label{generalsolitontempeq4}
\begin{aligned} 0 \, = & \, u'' (h^3) + u' \left( 3 h^2 h' - \frac{h^3 h''}{h'} + \frac{3 h^2}{h'} + \frac{\lambda h^4}{4 h'} \right) \\ & \qquad {} + u \left( - 3 h^2 h'' - 6 h - \frac{h^3 h'''}{h'} - \frac{3 h^2 h''}{(h')^2} - \frac{\lambda h^4 h''}{4 (h')^2} + \frac{h^3 (h'')^2}{(h')^2} \right).
\end{aligned}
\end{equation}
The next step is to eliminate $u = \sin 3 \theta$ from the above equation. Since $h' = \cos 3 \theta$, we have
\begin{equation} \label{ueq1}
u^2 \, = \, 1 - (h')^2.
\end{equation}
We can differentiate the above equation to get
\begin{equation} \label{ueq2}
u u' \, = \, - h' h''.
\end{equation}
Now we differentiate~\eqref{ueq2}, multiply both sides by $u^2$, and use both~\eqref{ueq1} and~\eqref{ueq2} again:
\begin{align*}
(u')^2 + u u'' & = \, - ((h'')^2 + h' h''') \\ u^2 ( (u')^2 + u u'' ) & = \, -u^2 ( (h'')^2 + h' h''' ) \\
(u u')^2 + u^3 u'' & = \, - (1 - (h')^2) ( (h'')^2 + h' h''' ) \\ (- h' h'')^2 + u^3 u'' & = \, - (h'')^2 - h' h''' + (h')^2 (h'')^2 + (h')^3 h'''.
\end{align*}
From the above we find
\begin{equation} \label{ueq3}
u^3 u'' \, = \, (h')^3 h''' - h' h''' - (h'')^2.
\end{equation}
We can now multiply equation~\eqref{generalsolitontempeq4} by $u^3$ and substitute~\eqref{ueq1},~\eqref{ueq2}, and~\eqref{ueq3} for $u^4 = (u^2)^2$, $u^3 u' = u^2 (u u')$, and $u^3 u'' = u^2 (u u'')$. We can then multiply through by $(h')^2$ to clear the denominators. This eliminates $u$ completely and leaves only a third order nonlinear (polynomial) ordinary differential equation for $h$. The result is:
\begin{equation} \label{generalsolitontempeq5}
\begin{aligned}
& h^3 (h')^3 h''' - h^3 h' h''' - 2 h^3 (h')^2 (h'')^2 + 3 h^2 (h')^4 h'' - 6 h (h')^2 + h^3 (h'')^2 - 3 h^2 h'' \\ & \qquad {} + 12 h (h')^4 - 6 h (h')^6 + \frac{\lambda}{4} h^4 (h')^2 h'' - \frac{\lambda}{4} h^4 h'' \, = \, 0.
\end{aligned}
\end{equation}
If one can solve this equation, then we also get the solution \emph{algebraically} for $u = \sin 3 \theta$ from~\eqref{ueq1} and for $k'$ from~\eqref{generalsolitontempeq1}. However, there does not appear to be an integrating factor for this differential equation and hence it is not clear if the general solution can be found explicitly, as is often (but not always) the case with cohomogeneity one solitons for geometric flows. See~\cite{Dancer/Wang:2009} for examples of cohomogeneity one Ricci solitons which were not exactly integrable, but where a dynamical systems analysis was possible.

\bibliographystyle{amsplain}
\bibliography{karigiannis-mckay-tsui}

\def\cprime{$'$}
\providecommand{\bysame}{\leavevmode\hbox to3em{\hrulefill}\thinspace}
\providecommand{\MR}{\relax\ifhmode\unskip\space\fi MR }
\providecommand{\MRhref}[2]{%
  \href{http://www.ams.org/mathscinet-getitem?mr=#1}{#2}
}
\providecommand{\href}[2]{#2}
\begin{thebibliography}{10}

\bibitem{Bar:1993}
Christian B{\"a}r, \emph{Real {K}illing spinors and holonomy}, Comm. Math.
  Phys. \textbf{154} (1993), no.~3, 509--521. \MR{1224089 (94i:53042)}

\bibitem{Boyer:Galicki:2008}
Charles~P. Boyer and Krzysztof Galicki, \emph{Sasakian geometry}, Oxford
  Mathematical Monographs, Oxford University Press, Oxford, 2008. \MR{2382957
  (2009c:53058)}

\bibitem{Bryant:1987b}
Robert Bryant, \emph{Metrics with exceptional holonomy}, Ann. of Math. (2)
  \textbf{126} (1987), no.~3, 525--576. \MR{MR916718 (89b:53084)}

\bibitem{Bryant:2006}
\bysame, \emph{Some remarks on {$G\sb 2$}-structures}, Proceedings of
  {G}\"okova {G}eometry-{T}opology {C}onference 2005, G\"okova
  Geometry/Topology Conference (GGT), G\"okova, 2006, pp.~75--109. \MR{2282011
  (2007k:53019)}

\bibitem{Bryant:Xu:2011}
Robert Bryant and Feng Xu, \emph{Laplacian flow for closed ${G}_2$-structures:
  Short time behavior}, arXiv:1101.2004.

\bibitem{Butruille:2005}
Jean-Baptiste Butruille, \emph{Classification des vari\'et\'es
  approximativement k\"ahleriennes homog\`enes}, Ann. Global Anal. Geom.
  \textbf{27} (2005), no.~3, 201--225. \MR{2158165 (2006f:53060)}

\bibitem{Cabrera:1996}
Francisco~M. Cabrera, \emph{On {R}iemannian manifolds with {$G\sb
  2$}-structure}, Boll. Un. Mat. Ital. A (7) \textbf{10} (1996), no.~1,
  99--112. \MR{1386249 (97j:53032)}

\bibitem{Cabrera:2006}
Francisco~Mart{\'{\i}}n Cabrera, \emph{{${\rm SU}(3)$}-structures on
  hypersurfaces of manifolds with {$G\sb 2$}-structure}, Monatsh. Math.
  \textbf{148} (2006), no.~1, 29--50. \MR{2229065 (2007b:53059)}

\bibitem{Chiossi:Salamon:2001}
Simon Chiossi and Simon Salamon, \emph{The intrinsic torsion of {$\rm SU(3)$}
  and {$G\sb 2$} structures}, Differential geometry, {V}alencia, 2001, World
  Sci. Publ., River Edge, NJ, 2002, pp.~115--133. \MR{1922042 (2003g:53030)}

\bibitem{Cleyton:Ivanov:2008}
Richard Cleyton and Stefan Ivanov, \emph{Curvature decomposition of {$G\sb
  2$}-manifolds}, J. Geom. Phys. \textbf{58} (2008), no.~10, 1429--1449.
  \MR{2453675 (2009m:53126)}

\bibitem{Cleyton:Swann:2002}
Richard Cleyton and Andrew Swann, \emph{Cohomogeneity-one {$G\sb
  2$}-structures}, J. Geom. Phys. \textbf{44} (2002), no.~2-3, 202--220.
  \MR{1969782 (2004a:53051)}

\bibitem{Dancer/Wang:2000}
Andrew Dancer and McKenzie Wang, \emph{The cohomogeneity one {E}instein
  equations from the {H}amiltonian viewpoint}, J. Reine Angew. Math.
  \textbf{524} (2000), 97--128. \MR{MR1770605 (2001e:53046)}

\bibitem{Dancer/Wang:2005}
\bysame, \emph{Superpotentials and the cohomogeneity one {E}instein equations},
  Comm. Math. Phys. \textbf{260} (2005), no.~1, 75--115. \MR{MR2175990
  (2006i:53066)}

\bibitem{Dancer/Wang:2009}
\bysame, \emph{Some new examples of non-{K}\"ahler {R}icci solitons}, Math.
  Res. Lett. \textbf{16} (2009), no.~2, 349--363. \MR{2496749 (2010e:53108)}

\bibitem{Fernandez:Gray:1982}
M.~Fern{\'a}ndez and A.~Gray, \emph{Riemannian manifolds with structure group
  {$G\sb{2}$}}, Ann. Mat. Pura Appl. (4) \textbf{132} (1982), 19--45 (1983).
  \MR{696037 (84e:53056)}

\bibitem{Fernandez:et:al:2008}
Marisa Fern{\'a}ndez, Stefan Ivanov, Vicente Mu{\~n}oz, and Luis Ugarte,
  \emph{Nearly hypo structures and compact nearly {K}\"ahler 6-manifolds with
  conical singularities}, J. Lond. Math. Soc. (2) \textbf{78} (2008), no.~3,
  580--604. \MR{2456893 (2009m:53061)}

\bibitem{Friedrich:et:al:1997}
Th. Friedrich, I.~Kath, A.~Moroianu, and U.~Semmelmann, \emph{On nearly
  parallel {$G\sb 2$}-structures}, J. Geom. Phys. \textbf{23} (1997), no.~3-4,
  259--286. \MR{1484591 (98j:53053)}

\bibitem{Friedrich:Ivanov:2002}
Thomas Friedrich and Stefan Ivanov, \emph{Parallel spinors and connections with
  skew-symmetric torsion in string theory}, Asian J. Math. \textbf{6} (2002),
  no.~2, 303--335. \MR{1928632 (2003m:53070)}

\bibitem{Hitchin:2000}
Nigel Hitchin, \emph{The geometry of three-forms in six dimensions}, J.
  Differential Geom. \textbf{55} (2000), no.~3, 547--576. \MR{MR1863733
  (2002m:53070)}

\bibitem{Joyce:2000}
Dominic~D. Joyce, \emph{Compact manifolds with special holonomy}, Oxford
  Mathematical Monographs, Oxford University Press, Oxford, 2000. \MR{MR1787733
  (2001k:53093)}

\bibitem{Karigiannis:2009}
Spiro Karigiannis, \emph{Flows of {$G\sb 2$}-structures. {I}}, Q. J. Math.
  \textbf{60} (2009), no.~4, 487--522. \MR{2559631}

\bibitem{Karigiannis:Notes}
\bysame, \emph{Some notes on {$G\sb 2$} and {${\rm Spin}(7)$} geometry}, Recent
  advances in geometric analysis, Adv. Lect. Math. (ALM), vol.~11, Int. Press,
  Somerville, MA, 2010, pp.~129--146. \MR{2648941}

\bibitem{Petersen/Wylie}
Peter Petersen and William Wylie, \emph{On gradient {R}icci solitons with
  symmetry}, Proc. Amer. Math. Soc. \textbf{137} (2009), no.~6, 2085--2092.
  \MR{2480290 (2010a:53073)}

\bibitem{Podesta:Spiro:2:2010}
Fabio Podest{\`a} and Andrea Spiro, \emph{Six-dimensional nearly {K}\"ahler
  manifolds of cohomogeneity one (ii)}, arXiv:1011.4681.

\bibitem{Podesta:Spiro:1:2010}
\bysame, \emph{Six-dimensional nearly {K}\"ahler manifolds of cohomogeneity
  one}, J. Geom. Phys. \textbf{60} (2010), no.~2, 156--164. \MR{2587385
  (2011c:53096)}

\bibitem{Reyes:Salamon:1999}
Ram{\'o}n Reyes~Carri{\'o}n and Simon Salamon, \emph{A survey of nearly
  {K}\"ahler manifolds}, Gac. R. Soc. Mat. Esp. \textbf{2} (1999), no.~1,
  40--49. \MR{1707644 (2000e:53032)}

\bibitem{Sternberg}
Shlomo Sternberg, \emph{Semi-riemann geometry and general relativity}, Lecture
  notes available at
  \url{http://www.math.harvard.edu/~shlomo/docs/semi_riemannian_geometry.pdf}.

\bibitem{Weiss:Witt:2010}
Hartmut Weiss and Frederik Witt, \emph{A heat flow for special metrics},
  arXiv:0912.0421.

\bibitem{Xu:Ye:2010}
Feng Xu and Rugang Ye, \emph{Existence, convergence and limit map of the
  {L}aplacian flow}, arXiv:0912.0074.

\end{thebibliography}
\end{document}